\documentclass[12pt]{amsart}
\usepackage{stmaryrd}
\usepackage{mathrsfs}
\usepackage{amsmath,amssymb}
\usepackage{amsfonts}
\usepackage{amsthm}
\usepackage{latexsym}
\usepackage{graphicx,color}
\def\R{\mathbb{R}}

\def\vv<#1>{\langle#1\rangle}

\def\1{\mathbf{1}}

\def\XXint#1#2{\setbox0=\hbox{$#1{#2}{\int}$}{#2}\kern-.5\wd0 }

\def\XXint#1#2#3{{\setbox0=\hbox{$#1{#2#3}{\int}$}
     \vcenter{\hbox{$#2#3$}}\kern-.5\wd0}}

\def\vv<#1>{{\left\langle#1\right\rangle}}

\def\wt{\widetilde}
\newtheorem{thm}{Theorem}[section]

\newtheorem{lem}{Lemma}[section]

\theoremstyle{definition}
\newtheorem{defn}{Definition}[section]
\theoremstyle{remark}

\numberwithin{equation}{section}

\begin{document}
\title{Faber-Krahn inequalities for first Dirichlet eigenvalues of combinatorial $p$-Laplacian on graphs with boundary}

\author{Wankai He}
\address{Department of Mathematics, Shantou University, Shantou, Guangdong, 515063, China}
\email{18wkhe@stu.edu.cn}
\author{Chengjie Yu$^1$}
\address{Department of Mathematics, Shantou University, Shantou, Guangdong, 515063, China}
\email{cjyu@stu.edu.cn}
\thanks{$^1$Research partially supported by GDNSF with contract no. 2025A1515011144 and 2026A1515012267.}
\renewcommand{\subjclassname}{%
  \textup{2020} Mathematics Subject Classification}
\date{}
\begin{abstract}
In this paper, we obtain sharp Faber-Krahn inequalities for the first Dirichlet eigenvalue of the combinatorial $p$-Laplacian on connected graphs with a fixed number of vertices or with a fixed number of edges. More precisely, we show that the minimum of the first $p$-Dirichlet eigenvalues of connected graphs with boundary that consist of $n$ vertices or $n$ edges is achieved only on the tadpole graph $T_{n,3}$ when $p>1$.
\end{abstract}
\maketitle
\markboth{He \& Yu}{Faber-Krahn inequality}
\section{Introduction}
In spectral geometry, the classical Faber-Krahn inequality (see \cite[P. 87]{Ch}) says that the minimum of the first Dirichlet eigenvalues for smooth bounded Euclidean domains with a fixed volume is only achieved on the round ball with the same volume. Experts also tried  to extend this important inequality to discrete setting which  makes it an important theme in extremal spectral graph theory.

In \cite{Fr}, Friedman considered Faber-Krahn inequality for regular trees. He formulated the problem of finding the Faber-Krahn inequality for domains in a homogeneous tree with a fixed total length and conjectured that the minimum is achieved by a geodesic ball in the homogeneous tree. In \cite{Pr}, Pruss disproved Friedman's conjecture. In a series of works \cite{Le97,Le02}, Leydold completely solved Friedman's problem.  In \cite{Fr2}, Friedman established Faber-Krahn inequalities for eigenvalues of the combinatorial Laplacian on graphs with a fixed number of vertices. The result was recently extended by the second named author and Yingtao Yu \cite{YY} to Steklov eigenvalues.

In \cite{KU}, Katsuda and Urakawa obtained a sharp Faber-Krahn inequality for the first Dirichlet eigenvalue of the normalized combinatorial Laplacian. In \cite{BL}, Bıyıkoğlu and Leydold formulated the general problem of finding graphs satisfying the so called Faber-Krahn property in a certain class of graphs and solved the problem for several classes of graphs such as trees with fixed numbers of vertices and boundary vertices, trees with fixed numbers of vertices and boundary vertices and with a fixed lower bound on degree of interior vertices, and trees with a given degree sequence. In \cite{ZZ12,ZZ13,ZZZ,WH,LLY}, the authors further considered the problem formulated in \cite{BL} for other classes of graphs.

In \cite{DS} and \cite{Ro}, the authors related the Faber-Krahn inequality on graphs to heat kernel bounds. In \cite{LLYZ}, the authors obtained interesting lower bounds for the first Dirichlet eigenvalue of a graph in terms of its diameter. In \cite{BaL} and \cite{HL}, the authors considered related estimates on Dirichlet eigenvalues for subgraphs of the integer lattice graph.

In this paper, motivated by the work of Katsuda-Urakawa \cite{KU}, we consider Faber-Krahn inequalities for the un-normalized combinatorial Laplacian on connected graphs with a fixed "volume". Let's first recall the main result in \cite{KU}.
\begin{thm}[Katsuda-Urakawa \cite{KU}]\label{thm-KU}
Let $G$ be a connected graph with boundary. Suppose that $|E(G)|=n\geq 4$. Then,
$$\lambda_1^{\rm nor}(G)\geq \lambda_1^{\rm nor }(T_{n,3})$$
and the equality holds if and only if $G=T_{n,3}$.
\end{thm}
Here a connected graph is considered as a graph with boundary by taking its pendant vertices, i.e. vertices of degree one, as the boundary vertices, and $T_{n,3}$ is a tadpole graph on $n$ vertices with the head of size $3$. For the definition of tadpole graph, see Section 2. Moreover, $\lambda_1^{\rm nor}(G)$ means the first Dirichlet eigenvalue for the normalized combinatorial Laplacian of $G$ which is defined as
$$\Delta^{\rm nor} f(x)=\frac{1}{\deg(x)}\sum_{y\sim x}(f(x)-f(y)).$$

In this paper, we consider Faber-Krahn inequality on graphs with boundary for the un-normalized combinatorial $p$-Laplacian:
$$\Delta_p f(x)=\sum_{y\sim x}|f(x)-f(y)|^{p-2}(f(x)-f(y))$$
for $p>1$. More precisely, we obtain the following two sharp Faber-Krahn inequalities for the first Dirichlet eigenvalue $\lambda_{1,p}$ of $\Delta_p$ on graphs with a fixed number of vertices or with a fixed number of edges.
\begin{thm}\label{thm-FK-V}
Let $G$ be a connected graph with boundary on $n\geq 4$ vertices and $p>1$. Then,
$$\lambda_{1,p}(G)\geq \lambda_{1,p}(T_{n,3})$$
and the equality holds if and only if $G=T_{n,3}$.
\end{thm}
\begin{thm}\label{thm-FK-E}
Let $G$ be a connected graph with  boundary that consists of $n\geq 4$ edges and $p>1$. Then, $$\lambda_{1,p}(G)\geq \lambda_{1,p}(T_{n,3})$$
and the equality holds if and only if $G=T_{n,3}$.
\end{thm}

Note that 
$$\lambda_{1,1}(G)=\lim_{p\to 1^+}\lambda_{1,p}(G).\ \mbox{(see \cite{Ge})}$$
So, the Faber-Krahn inequalities in Theorem \ref{thm-FK-V} and Theorem \ref{thm-FK-E} can be extended to $\lambda_{1,1}(G)$. However, the rigidity part does not hold in the case $p=1$. Indeed, by that 
$$\lambda_{1,1}(G)=h_D(G)$$ 
where $h_D(G)$ is Dirichlet Cheeger constant of $G$ (see \eqref{eq-Cheeger-constant}), it is clear that 
$$\lambda_{1,1}(G)\geq\frac{1}{n-1}=\lambda_{1,1}(T_{n,3}).$$
When $G$ is a connected graph with boundary on $n\geq 4$ vertices, the equality holds if and only if $G$ has only one pendant vertex. When $G$ consists of $n$ edges, we can characterize the rigidity more explicitly.
\begin{thm}\label{thm-FK-1}
Let $G$ be a connected graph consisting of $n\geq 4$ edges. Then, 
$$\lambda_{1,1}(G)\geq \frac{1}{n-1}$$
with equality if and only if $G=T_{n,i}$ for some $i$ with $3\leq i<n$.
\end{thm}

Our arguments to prove Theorem \ref{thm-FK-V} and Theorem \ref{thm-FK-E} are similar to that of Katsuda-Urakawa \cite{KU} by taking surgeries on graphs decreasing the Rayleigh quotient.  However, the surgeries we take are much simpler than that  of Katsuda-Urakawa \cite{KU}. In \cite{HY}, we gave a much simpler proof of Katsuda-Urakawa's result (Theorem \ref{thm-KU}) and extended it to normalized combinatorial $p$-Laplacian by using the same surgery.

The organization of the rest of paper is as follows: In Section 2, we introduce some basic definitions and notations and obtain some simple spectral properties for the first Dirichlet eigenvalue of tadpole graphs; In Section 3, we prove Theorem \ref{thm-FK-V} and Theorem \ref{thm-FK-E}.
\section{Preliminaries}
In this section, we introduce some basic definitions and notations, and obtain some simple spectral properties for the first Dirichlet eigenvalue of tadpole graphs.

Let $G$ be a nontrivial connected finite graph and
$$B(G):=\{x\in V(G)\ |\ \deg(x)=1\}$$
which is viewed as the boundary of $G$. The set
$$\Omega(G):=V(G)\setminus B(G)$$
is viewed as the interior of $G$. For $p>1$, the un-normalized combinatorial $p$-Laplacian on $G$ is defined as
$$\Delta_{p,G} f(x)=\sum_{y\sim x}|f(x)-f(y)|^{p-2}(f(x)-f(y)),\ \forall x\in V(G),$$
where $f\in \R^{V(G)}$. The un-normalized combinatorial $p$-Laplacian is also called the combinatorial $p$-Laplacian for simplicity.

A real number $\lambda$ is called  a $p$-Dirichlet eigenvalue of $G$ if the following Dirichlet boundary value problem:
$$\left\{\begin{array}{ll}\Delta_{p,G} f(x)=\lambda |f|^{p-2}f(x)&x\in \Omega(G)\\
f(x)=0&x\in B(G)
\end{array}\right.$$
has a nonzero solution $f$. In this case, $f$ is called a $p$-Dirichlet eigenfunction of $G$. We denote the smallest $p$-Dirichlet eigenvalue of $G$ as $\lambda_{1,p}(G)$ and call it the first $p$-Dirichlet eigenvalue of $G$. An eigenfunction of $\lambda_{1,p}(G)$ is called a  first $p$-Dirichlet eigenfunction of $G$. The  first $p$-Dirichlet eigenvalue of $G$ can be characterized by the minimum of the $p$-Rayleigh quotient:
\begin{equation}\label{eq-Rayleigh}
\lambda_{1,p}(G)=\min_{f\in C_B(G)\setminus\{0\}}R_{p,G}[f]
\end{equation}
and the minimum is only achieved by first $p$-Dirichlet eigenfunctions. Here
$$C_B(G)=\left\{f\in \R^{V(G)}\ \Big|\ f|_B\equiv 0\right\},$$
and
$$R_{p,G}[f]=\frac{\|df\|_{p,G}^p}{\|f\|_{p,G}^p}$$
with
$$\|df\|_{p,G}^p=\sum_{\{x,y\}\in E(G)}|f(x)-f(y)|^p$$
and
$$\|f\|_{p,G}^p=\sum_{x\in V(G)}|f|^p(x)=\sum_{x\in\Omega(G)}|f|^p(x)$$
since $f\in C_B(G).$ 

The definition of the $p$-Laplacian is subtle for $p=1$ (see \cite{Ch} for example).  However, the first $1$-Dirichlet eigenvalue $\lambda_{1,1}$ can be also characterized by minimum of the Rayleigh quotient \eqref{eq-Rayleigh} for $p=1$. It is also well-known (see \cite{Gr,HW} for example) that 
$$\lambda_{1,1}(G)=h_D(G)$$
where 
\begin{equation}\label{eq-Cheeger-constant}
h_D(G)=\inf_{\emptyset\neq U\subset \Omega(G)}\frac{|E(U,U^c)|}{|U|}
\end{equation}
is the Dirichlet Cheeger constant of $G$. For any nonempty set $U\subset \Omega(G)$, setting $f=\1_U$ (the characteristic function on $U$) in \eqref{eq-Rayleigh}, one obtains 
\begin{equation*}
\lambda_{1,p}(G)\leq h_D(G)
\end{equation*}
directly.

If $B(G)\neq \emptyset$ and $p>1$, then
the first $p$-Dirichlet eigenvalue $\lambda_1(G)$ is positive and the first $p$-Dirichlet eigenfunction $f$ is nonzero and does not change signs in $\Omega(G)$ (see \cite{HW}).  Without loss of generality, we can assume that $f$ is positive on $\Omega(G)$. In this case, we call $f$ a positive first $p$-Dirichlet eigenfunction of $G$. Moreover, the first $p$-Dirichlet eigenvalue is simple in the sense that first $p$-Dirichlet eigenfunction is unique up to a constant multiple (see \cite{HW}).

Next, we give the definition of tadpole graphs.
\begin{defn}
The wedge sum of  a cycle graph and a path graph on one end vertex of the path is called a tadpole graph. The other end vertex of the path is called the end vertex of the tadpole graph. The wedge sum vertex is called the neck vertex of the tadpole graph. The path and the cycle are called the tail and head of the tadpole graph respectively.
\end{defn}
For $n>i\geq 3$, we denote the tadpole graph on $n$ vertices with the head a cycle of length $i$ as $T_{n,i}$. More precisely, we can write $T_{n,i}$ as
\begin{equation}\label{eq-tadpole-graph}
T_{n,i}: t_n\sim t_{n-1}\sim \cdots\sim t_i\sim t_{i-1}\sim \cdots \sim t_2\sim t_1\sim t_i.
\end{equation}
The path
$$P:t_n\sim t_{n-1}\sim \cdots\sim t_i$$
is the tail of $T_{n,i}.$ The cycle
$$C:t_i\sim t_{i-1}\sim \cdots \sim t_2\sim t_1\sim t_i$$
is the head of $T_{n,i}$. The vertices $t_i$ and $t_n $ are the neck vertex and end vertex of $T_{n,i}$ respectively. 

Tadpole graphs have the following useful spectral properties.
\begin{lem}\label{lem-max-in-head}
 Let $n>i\geq 3$, $p>1$ and $f$ be a positive first $p$-Dirichlet eigenfunction of $T_{n,i}$. Then, $f$ does not achieve its maximum on the tail of $T_{n,i}$.
\end{lem}
\begin{proof}
We proceed by contradiction. Let $T_{n,i}$ be the tadpole graph as in \eqref{eq-tadpole-graph}. Suppose that $f$ achieves its maximum at some vertex $t_j$ with $i\leq j\leq n$. Let $\wt f\in \R^{V(T_{n,i})}$ be such that
$$\wt f(t_k)=\left\{\begin{array}{ll}f(t_j)& 1\leq k\leq j\\
f(t_k)&j<k\leq n.
\end{array}\right.$$
Then,
$$\|f\|_{p,T_{n,i}}^p=\sum_{k=1}^nf^p(t_k)\leq \sum_{k=1}^n\wt f^p(t_k)=\|\wt f\|_{p,T_{n,i}}^p$$
and
\begin{equation*}
\begin{split}
\|df\|_{p,T_{n,i}}^p=&\sum_{k=1}^{n-1}|f(t_{k+1})-f(t_{k})|^p+|f(t_1)-f(t_i)|^p\\
\geq&\sum_{k=j}^{n-1}|f(t_{k+1})-f(t_{k})|^p\\
=&\|d\wt f\|_{p,T_{n,i}}^p.
\end{split}
\end{equation*}
Therefore, by \eqref{eq-Rayleigh},
$$\lambda_{1,p}(T_{n,i})=R_{p,T_{n,i}}[f]\geq R_{p,T_{n,i}}[\wt f]\geq \lambda_{1,p}(T_{n,i})$$
and hence $\wt f$ is also a first $p$-Dirichlet eigenfunction of $T_{n,i}$. Then,
\begin{equation*}
\begin{split}
0<&\lambda_{1,p}(T_{n,i})\wt f^{p-1}(t_1)\\
=&\Delta_{p,T_{n,i}}\wt f(t_1)\\
=&|\wt f(t_1)-\wt f(t_i)|^{p-2}(\wt f(t_1)-\wt f(t_i))+|\wt f(t_1)-\wt f(t_2)|^{p-2}(\wt f(t_1)-\wt f(t_2))\\
=&0
\end{split}
\end{equation*}
which is ridiculous. This completes the proof of the lemma.
\end{proof}

We next show that the first Dirichlet eigenvalue of $T_{n,4}$ is greater than that of $T_{n,3}$.
\begin{lem}\label{lem-comparison-tadpole}
For $n\geq 5$ and $p>1$,
$$\lambda_{1,p}(T_{n,4})>\lambda_{1,p}(T_{n,3}).$$
\end{lem}
\begin{proof}
Let $f$ be a positive first $p$-Dirichlet eigenfunction of
$$T_{n,4}: v_n\sim v_{n-1}\sim \cdots\sim v_4\sim v_3\sim v_2\sim v_1\sim v_4.$$
By symmetry and that $\lambda_{1,p}(T_{n,4})$ is of multiplicity one, we know that
\begin{equation}\label{eq-f-v1=v3}
f(v_1)=f(v_3).
\end{equation}
Let $m$ be a maximum vertex of $f$. By Lemma \ref{lem-max-in-head}, $m$ is not in the tail of $T_{n,4}$. So, without loss of generality, $m=v_3$ or $m=v_2$. Let
$$T_{n,3}: u_n\sim u_{n-1}\sim\cdots\sim u_{3}\sim u_2\sim u_1\sim u_3.$$

When $m=v_3$, let $\wt f\in \R^{V(T_{n,3})}$ be such that
\begin{equation*}
\wt f(u_k)=\left\{\begin{array}{ll}f(v_k)&3\leq k\leq n\\
f(m)&1\leq k<3.
\end{array}\right.
\end{equation*}
Then,
\begin{equation*}
\|f\|_{p,T_{n,4}}^p=\sum_{k=1}^nf^p(v_k)\leq \sum_{k=1}^n\wt f^p(u_k)=\|\wt f\|_{p,T_{n,3}}^p
\end{equation*}
and
\begin{equation*}
\begin{split}
\|df\|_{p,T_{n,4}}^p=&\sum_{k=1}^{n-1}|f(v_{k+1})-f(v_k)|^p+|f(v_1)-f(v_4)|^p\\
\geq&\sum_{k=3}^{n-1}|f(v_{k+1})-f(v_k)|^p\\
=&\|d\wt f\|_{p,T_{n,3}}^p.
\end{split}
\end{equation*}
Hence, by \eqref{eq-Rayleigh},
$$\lambda_{1,p}(T_{n,4})=R_{p,T_{n,4}}[f]\geq R_{p,T_{n,3}}[\wt f]>\lambda_{1,p}(T_{n,3}).$$
The last inequality is strict because $\wt f$ is not a first $p$-Dirichlet eigenfunction for $T_{n,3}$ by Lemma \ref{lem-max-in-head}.

When $m=v_2$, let $\wt f\in \R^{V(T_{n,3})}$ be such that
\begin{equation*}
\wt f(u_k)=f(v_k)
\end{equation*}
for $k=1,2,\cdots,n$. By that
\begin{equation*}
\begin{split}
0<&\lambda_{1,p}(T_{n,4})f^{p-1}(v_1)\\
=&\Delta_{p,T_{n,4}} f(v_1)\\
=&|f(v_1)-f(v_4)|^{p-2}(f(v_1)-f(v_4))+|f(v_1)-f(v_2)|^{p-2}(f(v_1)-f(v_2))\\
\leq& |f(v_1)-f(v_4)|^{p-2}(f(v_1)-f(v_4))
\end{split}
\end{equation*}
we have $f(v_1)-f(v_4)>0$. Moreover, by \eqref{eq-f-v1=v3}, 
\begin{equation*}
\begin{split}
\|df\|_{p,T_{n,4}}^p=&\sum_{k=1}^{n-1}|f(v_{k+1})-f(v_k)|^{p}+|f(v_1)-f(v_4)|^p\\
>&\sum_{k=1}^{n-1}|f(v_{k+1})-f(v_k)|^p+|f(v_3)-f(v_1)|^p\\
=&\|d\wt f\|_{p,T_{n,3}}^p,\\
\end{split}
\end{equation*}
and it is clear that
$$\|f\|_{p,T_{n,4}}^p=\|\wt f\|_{p,T_{n,3}}^p.$$
So, by \eqref{eq-Rayleigh},
$$\lambda_{1,p}(T_{n,4})=R_{p,T_{n,4}}[f]>R_{p,T_{n,3}}[\wt f]\geq \lambda_{1,p}(T_{n,3}).$$
This completes the proof of the lemma.
\end{proof}
Finally, we compare the first $p$-Dirichlet eigenvalue of a path graph on $n$ or $n+1$ vertices to that of $T_{n,3}$. We denote the path graph on $n$ vertices as $P_n$.
\begin{lem}\label{lem-comparison-path}
For any $n\geq 4$ and $p>1$,
$$\lambda_{1,p}(P_n)>\lambda_{1,p}(P_{n+1})>\lambda_{1,p}(T_{n,3}).$$
\end{lem}
\begin{proof}
Let
$$P_n:v_n\sim v_{n-1}\sim\cdots\sim v_2\sim v_1$$
and $f$ be a positive first $p$-Dirichlet eigenfunction of $P_n$. Let
$$P_{n+1}: u_{n}\sim u_{n-1}\sim\cdots\sim u_2\sim u_1\sim u_0$$
and
$$\wt f(u_k)=\left\{\begin{array}{ll}f(v_k)&1\leq k\leq n\\
0&k=0.
\end{array}\right.$$
Then, by \eqref{eq-Rayleigh},
$$\lambda_{1,p}(P_n)=R_{p,P_n}[f]=R_{p,P_{n+1}}[\wt f]>\lambda_{1,p}(P_{n+1}).$$
The last inequality is strict because $\wt f$ is not a $p$-Dirichlet eigenfunction of $P_{n+1}$ since $\wt f$ is not everywhere positive in the interior of $P_{n+1}$.

Moreover, let $g$ be a positive first $p$-Dirichlet eigenfunction of $P_{n+1}$ and
$$T_{n,3}: w_n\sim w_{n-1}\sim\cdots\sim w_3\sim w_2\sim w_1\sim w_3.$$
When $n\geq 5$, let
\begin{equation}
\wt g(w_k)=\left\{\begin{array}{ll}g(u_k)&k\geq 3\\
g(u_3)&k=1,2.
\end{array}\right.
\end{equation}
By that $g(u_3)\geq g(u_2)>g(u_1)>0$, we have
\begin{equation*}
\|g\|_{p,P_{n+1}}^p=\sum_{k=1}^ng^p(u_k)<\sum_{k=1}^n \wt g^p(w_k)=\|\wt g\|_{p,T_{n,3}}^p
\end{equation*}
and
\begin{equation*}
\begin{split}
\|dg\|_{p,P_{n+1}}^p=&\sum_{k=1}^{n}|g(u_{k})-g(u_{k-1})|^p\\
>&\sum_{k=4}^{n}|g(u_k)-g(u_{k-1})|^p\\
=&\|d\wt g\|_{p,T_{n,3}}^p.
\end{split}
\end{equation*}
So, by \eqref{eq-Rayleigh},
$$\lambda_{1,p}(P_{n+1})=R_{p,P_{n+1}}(g)>R_{p,T_{n,3}}(\wt g)\geq \lambda_{1,p}(T_{n,3}).$$

When $n=4$, let
\begin{equation*}
\wt g(w_k)=g(u_k)\ \mbox{for $k=1,2,3,4.$}
\end{equation*}
Note that $g(u_3)=g(u_1)$ by symmetry. So,
\begin{equation*}
\begin{split}
\|dg\|_{p,P_5}^p=&\sum_{k=1}^4|g(u_k)-g(u_{k-1})|^p\\
>&\sum_{k=2}^4|g(u_k)-g(u_{k-1})|^p\\
=&\sum_{k=2}^4|\wt g(w_k)-\wt g(w_{k-1})|^p+|\wt g(w_1)-\wt g(w_3)|^p\\
=&\|d\wt g\|_{p,T_{4,3}}^p,
\end{split}
\end{equation*}
and it is clear that
$$\|g\|_{p,P_5}^p=\|\wt g\|_{p,T_{4,3}}^p.$$
Thus, by \eqref{eq-Rayleigh},
$$\lambda_{1,p}(P_{5})=R_{p,P_{5}}[g]>R_{p,T_{4,3}}[\wt g])\geq\lambda_{1,p}(T_{4,3}).$$
This completes the proof of the lemma.
\end{proof}
We woul like to mention that the spectral properties above for tadpole graphs are mainly given in  Katsuda-Urakawa \cite{KU} for normalized combinatorial Laplacian and the arguments above are similar to that of Katusda-Urakawa \cite{KU}. Here, we just modify them for the un-normalized combinatorial $p$-Laplacian.

\section{Proofs of main results}
In this section, we prove the main results. We first prove Theorem \ref{thm-FK-V}.
\begin{proof}[Proof of Theorem \ref{thm-FK-V}]
Let $f$ be a positive first $p$-Dirichlet eigenfunction of $G$. Suppose that
$$f(m)=\max_{v\in V(G)}f(v).$$
Let
$$P:v_n\sim v_{n-1}\sim \cdots\sim v_i=m$$
be a shortest path joining the boundary vertex $v_n$ to $m=v_i$ and suppose that
$v_1,v_2,\cdots,v_n$ are the vertices of $G$. Because $f(m)>0$ and $P$ is a shortest path joining $v_n$ and $v_i$, we know that $i\geq 2$.

When $i\geq 3$, let
$$T_{n,3}:u_n\sim u_{n-1}\sim \cdots \sim u_i\sim\cdots \sim u_3\sim u_2\sim u_1\sim u_3,$$
and
$$\wt f(u_k)=\left\{\begin{array}{ll}f(v_k)&i\leq k\leq n\\
f(m)&1\leq k<i.
\end{array}\right.$$
Then,
\begin{equation*}
\|f\|_{p,G}^{p}=\sum_{k=1}^nf^p(v_k)\leq \sum_{k=1}^n\wt f^p(u_k)=\|\wt f\|_{p,T_{n,3}}^p
\end{equation*}
and
\begin{equation*}
\begin{split}
\|df\|_{p,G}^p\geq \sum_{k=i}^{n-1} |f(v_{k+1})-f(v_{k})|^p=\sum_{k=i}^{n-1} |\wt f(v_{k+1})-\wt f(v_{k})\|^p=\|d\wt f\|_{p,T_{n,3}}^p.
\end{split}
\end{equation*}
So, by \eqref{eq-Rayleigh},
$$\lambda_{1,p}(G)=R_{p,G}[f]\geq R_{p,T_{n,3}}[\wt f]>\lambda_{1,p}(T_{n,3}).$$
The last inequality is strict because $\wt f$ is not a positive first $p$-Dirichlet eigenfunction of $T_{n,3}$ by Lemma \ref{lem-max-in-head}.

When $i=2$, we know that $v_1\sim v_2$ in $G$ because $v_2$ is not a boundary vertex. Moreover, $v_1$ can not be adjacent to $v_j$ with $j\geq 5$ because $P$ is the shortest path joining $v_n$ to $v_2$. Then, by the adjacency of $v_1$ to $v_3$ and $v_4$, we  only have the following four cases:

(1) $v_1\not\sim v_3$ and $v_1\not\sim v_4$: In this case, $G=P_n$, by Lemma \ref{lem-comparison-path}, $$\lambda_{1,p}(G)>\lambda_{1,p}(T_{n,3}).$$

(2) $v_1\not\sim v_3$ and $v_1\sim v_4$: In this case, $G=T_{n,4}$, by Lemma \ref{lem-comparison-tadpole}, $$\lambda_{1,p}(G)>\lambda_{1,p}(T_{n,3}).$$

(3) $v_1\sim v_3$ and $v_1\sim v_4$: In this case, by that
\begin{equation*}
\begin{split}
0<&\lambda_1(G) f^{p-1}(v_1)\\
=&\Delta_{p,G} f(v_1)\\
=&|f(v_1)-f(v_2)|^{p-2}(f(v_1)-f(v_2))\\
&+|f(v_1)-f(v_3)|^{p-2}(f(v_1)-f(v_3))+|f(v_1)-f(v_4)|^{p-2}(f(v_1)-f(v_4))\\
\leq&|f(v_1)-f(v_3)|^{p-2}(f(v_1)-f(v_3))+|f(v_1)-f(v_4)|^{p-2}(f(v_1)-f(v_4)),\\
\end{split}
\end{equation*}
we have $f(v_1)-f(v_3)>0$ or $f(v_1)-f(v_4)>0$.

If $f(v_1)-f(v_3)>0$, let $G'=G-\{v_1,v_3\}$. It is clear that $G'=T_{n,4}$. So, by \eqref{eq-Rayleigh} and Lemma \ref{lem-comparison-tadpole},
$$\lambda_{1,p}(G)=R_{p,G}[f]>R_{p,G'}[f]\geq \lambda_{1,p}(G')=\lambda_{1,p}(T_{n,4})>\lambda_{1,p}(T_{n,3}).$$

If $f(v_1)-f(v_4)>0$, let $G'=G-\{v_1,v_4\}.$ Then $G'=T_{n,3}$. So, by \eqref{eq-Rayleigh},
$$\lambda_{1,p}(G)=R_{p,G}[f]>R_{p,G'}[f]\geq \lambda_{1,p}(G')=\lambda_{1,p}(T_{n,3}).$$

(4) $v_1\sim v_3$ and $v_1\not\sim v_4$: In this case, $G=T_{n,3}$.

This completes the proof of the theorem.
\end{proof}

Finally, we prove Theorem \ref{thm-FK-E}.
\begin{proof}[Proof of Theorem \ref{thm-FK-E}]
If $G$ is a path graph, by Lemma \ref{lem-comparison-path}, $$\lambda_{1,p}(G)=\lambda_{1,p}(P_{n+1})>\lambda_{1,p}(T_{n,3}).$$
 So, we assume that $G$ is not a path graph and hence there exists a vertex $v\in V(G)$ such that  $\deg v\geq 3$. Then, by that
\begin{equation*}
\begin{split}
2n>\sum_{x\in\Omega(G)}\deg(x)=\deg (v)+\sum_{\Omega(G)\setminus\{v\}}\deg(x)\geq 3+2(|\Omega(G)|-1),
\end{split}
\end{equation*}
we have
\begin{equation}\label{eq-Omega}
|\Omega(G)|\leq n-1.
\end{equation}

Let $f$ be a positive first $p$-Dirichlet eigenfunction of $G$ and $m\in V(G)$ be a maximum point of $f$. Let
$$P: v_n\sim v_{n-1}\sim \cdots\sim v_i=m$$
be a shortest path joining the boundary vertex $v_n$ to $v_i=m$. Note that if $|E(G)|-|E(P)|=1$, then $G={P_{n+1}}$ because $m$ is not boundary vertex and $P$ is a shortest path joining $v_n$ to $m$. Thus $|E(G)|-|E(P)|\geq 2$.

When $|E(G)|-|E(P)|\geq 3$, we have $i\geq 3$. Let
$$T_{n,3}:u_n\sim u_{n-1}\sim\cdots\sim u_{i}\sim \cdots\sim u_3\sim u_2\sim u_1\sim u_3$$
and
$$\wt f(u_k)=\left\{\begin{array}{ll}f(v_k)&i\leq k\leq n\\
f(m)&1\leq k<i.
\end{array}\right.$$
Then, by \eqref{eq-Omega}, we have
\begin{equation*}
\begin{split}
\|f\|_{p,G}^p=&\sum_{k=i}^{n-1}f^p(v_k)+\sum_{x\in\Omega(G)\setminus\{v_i,v_{i+1},\cdots,v_{n-1}\}}f^p(x)\\
\leq&\sum_{k=i}^{n-1} \wt f^p(u_k)+\left(|\Omega(G)|-(n-i)\right)f^p(m)\\
\leq&\sum_{k=i}^{n-1} \wt f^p(u_k)+(i-1)f^p(m)\\
=&\|\wt f\|_{p,T_{n,3}}^{p}
\end{split}
\end{equation*}
and
\begin{equation*}
\|df\|_{p,G}^p\geq \sum_{k=i}^{n-1}|f(v_{k+1})-f(v_k)|^p=\sum_{k=i}^{n-1}|\wt f(u_{k+1})-\wt f(u_k)|^p=\|d \wt f\|_{p,T_{n,3}}^p.
\end{equation*}
So, by \eqref{eq-Rayleigh},
\begin{equation}
\lambda_{1,p}(G)=R_{p,G}[f]\geq R_{p,T_{n,3}}[\wt f]> \lambda_{1,p}(T_{n,3}).
\end{equation}
The last inequality is strict because $\wt f$ is not a positive first $p$-Dirichlet eigenfunction of $T_{n,3}$ by Lemma \ref{lem-max-in-head}.

When $|E(G)|-|E(P)|=2$, we have $i=2$. Because $v_2$ is not a boundary vertex and $P$ is a shortest path joining $v_n$ and $v_2$, there is another vertex $v_1$ adjacent to $v_2$. If there is no other vertex of $G$, then $G$ is a connected graph on $n$ vertices with nonempty boundary. By Theorem \ref{thm-FK-V}, we know that
$$\lambda_{1,p}(G)\geq \lambda_{1,p}(T_{n,3})$$
with equality if and only if $G=T_{n,3}$. Otherwise, let $v_0$ be another vertex of $G$ which should be one of the end points of the remaining edge. Because $G\neq P_{n+1}$, $v_0\sim v_j$ for some $j=2,3,\cdots,n-1$. Let $G'=G-\{v_0\}$. Then $G'=P_n$ and
by \eqref{eq-Rayleigh} and Lemma \ref{lem-comparison-path},
$$\lambda_{1,p}(G)=R_{p,G}[f]>R_{p,G'}[f]\geq \lambda_{1,p}(P_n)>\lambda_{1,p}({T_{n,3}}).$$
This completes the proof of the theorem.
\end{proof}
Finally, we come to prove Theorem \ref{thm-FK-1}.
\begin{proof}[Proof of Theorem \ref{thm-FK-1}]
For any nonempty subset $U$ of $\Omega(G)$, note that
\begin{equation}\label{eq-U-1}
2n>\sum_{x\in U}\deg(x)\geq 2|U|.
\end{equation}
So, $|U|\leq n-1$. Thus, 
\begin{equation}\label{eq-U-2}
\frac{|E(U,U^c)|}{|U|}\geq \frac{1}{|U|}\geq \frac{1}{n-1},
\end{equation}
and
$$\lambda_{1,1}(G)=h_{D}(G)\geq\frac{1}{n-1}.$$
If the equality holds,  let $U\subset \Omega(G)$ be such that 
$$\frac{|E(U,U^c)|}{|U|}=h_D(G)=\frac{1}{n-1}.$$
Then, by \eqref{eq-U-1} and \eqref{eq-U-2}, we know that $|E(U,U^c)|=1$ and $|U|=n-1$. Then, by that 
\begin{equation}\label{eq-B}
\begin{split}
1\leq&|B(G)|\\
=&2n-\sum_{x\in \Omega(G)}\deg(x)\\
\leq&2n-2(\Omega(G)-|U|)-\sum_{x\in U}\deg(x)\\
\leq& 2-2(\Omega(G)-|U|).\\
\end{split}
\end{equation}
We have $U=\Omega(G)$ and $G$ has at most two pendant vertices. 

If $|B(G)|=2$, then by \eqref{eq-B}, $\deg(x)=2$ for any $x\in\Omega(G)$. This implies that $G=P_{n+1}$. However, 
$$h_D(P_{n+1})=\frac{2}{n-2}>\frac{1}{n-1}.$$
Thus, $|B(G)|=1$. By \eqref{eq-B},  
$$2n-2\leq \sum_{x\in \Omega(G)}\deg(x)=2n-1.$$
This implies there is only one interior vertex of degree $3$ and the other interior vertices are all of degree $2$. Let $v_n$ be the boundary vertex of $G$, and 
$$P:v_n\sim v_{n-1}\sim \cdots\sim v_i$$ 
be the shortest path from $v_n$ to the interior vertex $v_i$ of degree $3$. Let 
$$G'=G-v_n-v_{n-1}-\cdots-v_{i+1}.$$
Then, every vertex $G'$ is of degree $2$ which implies that $G'$ is a cycle of length $i$. Thus $G=T_{n,i}$. Conversely, it is clear that 
$$h_D(T_{n,i})=\frac{1}{n-1}.$$
This completes the proof of the theorem.
\end{proof}

\end{document}